\documentclass[11pt]{amsart}

\usepackage{amsmath,amssymb,latexsym,soul,cite,mathrsfs}

\usepackage{color,enumitem,graphicx}
\usepackage[colorlinks=true,urlcolor=blue,
citecolor=red,linkcolor=blue,linktocpage,pdfpagelabels,
bookmarksnumbered,bookmarksopen]{hyperref}
\usepackage[english]{babel}

\usepackage[left=2.9cm,right=2.9cm,top=2.8cm,bottom=2.8cm]{geometry}
\usepackage[hyperpageref]{backref}

\usepackage[colorinlistoftodos]{todonotes}
\makeatletter
\providecommand\@dotsep{5}
\def\listtodoname{List of Todos}
\def\listoftodos{\@starttoc{tdo}\listtodoname}
\makeatother

\numberwithin{equation}{section}

\newtheorem{theorem}{Theorem}[section]
\newtheorem{proposition}[theorem]{Proposition}
\newtheorem{lemma}[theorem]{Lemma}
\newtheorem{corollary}[theorem]{Corollary}

\newtheorem{claim}[theorem]{Claim}

\begin{document}

\title[Existence of solution for a class of nonlocal problem ...]
{Existence of solution for a class of nonlocal problem via dynamical methods}
\author{Claudianor O. Alves$^{*}$ and Tahir Boudjeriou}
\address[Claudianor O. Alves ]
{\newline\indent Unidade Acad\^emica de Matem\'atica
	\newline\indent 
	Universidade Federal de Campina Grande 
\newline\indent
e-mail:coalves@mat.ufcg.edu.br
\newline\indent	
58429-970, Campina Grande - PB, Brazil} 

\address[Tahir Boudjeriou]
{\newline\indent Department of Mathematics
	\newline\indent 
	University of Bejaia- Algeria 
\newline\indent
e-mail:re.tahar@yahoo.com}

\pretolerance10000


\begin{abstract}
	\noindent In this paper we use the dynamical methods to establish the existence of nontrivial solution for a class of nonlocal problem of the type 
$$
\left\{\begin{array}{l}
-a\left(x,\int_{\Omega}g(u)\,dx \right)\Delta u =f(u),  \quad x \in \Omega \\
u=0, \hspace{2 cm} x \in \partial \Omega,  
\end{array}\right.
\leqno{(P)}
$$
where $\Omega \subset \mathbb{R}^N \, ( N \geq 2)$ is a smooth bounded domain and $a:\overline{\Omega} \times \mathbb{R} \to \mathbb{R}$ and $g,f: \mathbb{R} \to \mathbb{R}$ are $C^1$-functions that satisfy some technical conditions.  
\end{abstract}

\thanks{$^{*}$ C. O. Alves was partially
supported by  CNPq/Brazil 304804/2017-7 }

\maketitle	

\section{Introduction}

This paper concerns with the existence of nontrivial solutions for a nonlocal problem of the type 
$$
\left\{\begin{array}{l}
-a\left(x,\int_{\Omega}g(u)\,dx \right)\Delta u =f(u),  \quad x \in \Omega \\
u=0, \hspace{2 cm} x \in \partial \Omega  
\end{array}\right.
\leqno{(P)}
$$
where $\Omega \subset \mathbb{R}^N \, ( N \geq 2)$ is a smooth bounded domain and $a:\overline{\Omega} \times \mathbb{R} \to \mathbb{R}$ and $g,f: \mathbb{R} \to \mathbb{R}$ are $C^1$-functions that satisfy some technical conditions, which will be mentioned later on. 

Hereafter, we will assume that there exist $a_0,K>0$ such that 
\noindent   
$$
a_0 \leq a(x,t), \quad \forall (x,t) \in \overline{\Omega} \times \mathbb{R} \leqno{(a_1)}
$$
and  
$$
a(x,t)=1, \quad \forall x \in \overline{\Omega} \quad \mbox{and} \quad |t| \geq K. \leqno{(a_2)}
$$

With relation to the functions $f,g:\mathbb{R} \to \mathbb{R}$ we  assume the conditions  below that can depend on the dimension $N$. \\

\noindent {\bf Condition $(H)$}: $g(t) \geq 0$ for all $t \in \mathbb{R}$ and there is $C>0$ such that 
$$
\max\{|f(t)|^{2},|f(t)t|\} \leq C(g(t)+1), \quad \forall t \in \mathbb{R}.
$$

\noindent {\bf Condition $(g)$}:\\

\noindent {\bf Dimension $N=2$}:  
The functions ${g}$ and $g'$ have an exponential subcritical growth at infinity, that is, 
$$
\lim_{|t| \rightarrow \infty} \frac{g(t)}{e^{\beta t^2}} = \lim_{|t| \rightarrow \infty} \frac{g'(t)}{e^{\beta t^2}}=
0\quad \mbox{for all}\quad \beta > 0. \leqno{(I)}
$$
This condition combined with $(H)$ ensures that $f$ also has an exponential subcritical growth at infinity.

In dimension two an important tool is the Trudinger and  Moser inequality that states the following: For all $u\in W_{0}^{1,N}(\Omega)$ ($N\geq 2$), 
\begin{equation}\label{M}
e^{\alpha |u|^{\frac{N}{N-1}}}\in L^{1}(\Omega),\;\;\forall \alpha >0 \quad (\mbox{see}  \mbox{\cite{M}} )
\end{equation}
and there exist positive constants $C_{N}$ and $\alpha_{N}$ such that
\begin{equation}\label{T}
\sup_{||u||_{W_{0}^{1,N}(\Omega)}\leq 1}\int_{\Omega} e^{\alpha |u(x)|^{\frac{N}{N-1}}}\,dx\leq C_{N}, \;\forall \alpha \leq \alpha_{N}\;(\alpha_{N}:=\text{vol. unit sphere}). \quad (\mbox{see}  \mbox{\cite{T}} )
\end{equation}

\noindent {\bf Dimension $N \geq 3$}:  \\

\noindent There are $C_1>0$ and $q \in [2,2^{*}]$ such that 
$$
|g(t)| \leq C_1(1+|t|^{q}), \quad \quad \forall t \in \mathbb{R},
$$
where $2^*=\frac{2N}{N-2}$. \\

\noindent {\bf Condition $(f_1)$}:\\

\noindent {\bf Dimension $N = 2$}: 
$$
\lim_{|t| \rightarrow \infty} \frac{f'(t)}{e^{\beta t^2}}=0\quad \mbox{for all}\quad \beta > 0.
$$

\noindent {\bf Dimension $N \geq 3$}:  \\

\noindent There are $C_2>0$ and $p \in (1,q/2)$ such that 
$$
|f(t)| \leq C_2(1+|t|^{p}), \quad \quad \forall t \in \mathbb{R}.
$$

\noindent {\bf Condition $(f_2)$:}
$$
\lim_{t\to 0}\frac{f(t)}{t}=0. 
$$

\noindent {\bf Condition $(f_3)$:} \\
\noindent There exists $\gamma >0$ such that 
$$ 
f(s)s\geq (2+\gamma)F(s)>0, \;\forall s\in \mathbb{R} \setminus \{0\} \quad (\mbox{Ambrosetti-Rabinowitz condition}). 
$$
\mbox{}\\
The condition $(f_3)$ implies that there are $c_1,c_2 \geq 0$ such that
$$
F(t) \geq c_1|t|^{2+\gamma}-c_2, \quad \forall t \in \mathbb{R}.
$$
Thus,
\begin{equation} \label{NEWEQ1}
|t|^{2+\gamma} \leq c_3f(t)t+c_4 \quad \forall t \in \mathbb{R},
\end{equation}
for some constants $c_3,c_4 \geq 0$. 

\vspace{1 cm}

The interest by problem $(P)$ comes from the articles of Alves and Covei \cite{AlvesCovei}, Alves, Chipot and Corr\^ea \cite{AlvesCorreaChipot}, Chipot and Lovat \cite{ChipotLovat1,ChipotLovat2}, Chipot and Rodrigues \cite{ChipotRodigues}, Chipot and Corr\^ea \cite{ChipotCorrea} and Corr\^ea, Menezes and Ferreira \cite{CorreaMenezesFerreira} and  Gasi\'nki and Santos J\'unior \cite{joao}, where the authors study classes of nonlocal problems motivated by the fact
that they appear in some applied mathematics areas. More exactly, it is pointed out in the paper \cite[see pp. 4619-4620]{ChipotLovat1}, that
if $g(t)=t$, the solution $u$ of the problem $(P)$ could describe the density of a population subject
to spreading where the diffusion coefficient $a$ is supposed to depend on the entire population in the domain rather than
on the local density. Moreover, in \cite{ChipotLovat1}, the authors have mentioned that the importance of such model lies in the fact that measurements that serve to determine physical constants are not made at a point but represent an average in a neighborhood of a point so that these physical constants depend on local averages.

In what follows, in order to apply our approach we will rewrite problem $(P)$ in the form

$$
\left\{\begin{array}{l}
-\Delta u =f(u)+\Psi\left(x,u,\int_{\Omega}g(u)\,dx \right),  \quad x \in \Omega \\
u=0, \hspace{2 cm} x \in \partial \Omega,  
\end{array}\right.
\leqno{(P)'}
$$
where
\begin{equation} \label{psi}
\Psi(x,t,z)=\left(\frac{1}{a(x,z)}-1\right)f(t), \quad \forall (x,t,z) \in \overline{\Omega} \times \mathbb{R} \times \mathbb{R}. 
\end{equation}

From $(a_2)$,
\begin{equation} \label{eq1o}
\Psi(x,t,z)=0, \quad \forall (x,t) \in \overline{\Omega} \times \mathbb{R} \quad \mbox{and} \quad |z| \geq K. 
\end{equation}

Since the problem $(P)'$ is not variational we will apply dynamical methods to find a nontrivial solution for $(P)'$. This method consists in studying the parabolic problem associated with $(P)'$ given by    
\begin{equation}\label{1}
\left\{\begin{array}{l}
u_{t}-\Delta u =f(u)+\Psi\left(x,u,\int_{\Omega}g(u)\,dx \right)\;\text{in}\;\;\; \Omega \times (0,T), \\
u=0\;\;\;\;\hspace{2cm} \text{on}\;\partial \Omega \times [0,T], \\
u_{|_{t=0}}=u_{0},\hspace{1.5cm}\; \; x\in \Omega, 
\end{array}\right.
\end{equation}
for a special choice of $u_{0} \in  H_{0}^{1}(\Omega)$ that will guarantee the existence of a solution $u:[0,+\infty) \to H^{1}_{0}(\Omega)$ such that for some $t_n \to +\infty$, there is a nontrivial solution $u_s$ of $(P)'$ such that $u(t_n) \to u_s$ in  $H^{1}_{0}(\Omega)$ when $n\to +\infty$.

This type of approach has been considered by Quittner \cite{Q} to establish the existence of a nontrivial solution for the follow class of non variational elliptic problem
$$
\left\{\begin{array}{l}
-\Delta u =f(u)+\tilde{f}(x,u,\nabla u),  \quad x \in \Omega \\
u=0, \hspace{2 cm} x \in \partial \Omega.  
\end{array}\right.
\leqno{(P_1)}
$$
In \cite{Q1}, Quittner also used the dynamical methods to prove the existence of signed solution for the following class of problem 
$$
\left\{\begin{array}{l}
-\Delta u =u_{+}^p-u_{-}^{q},  \quad x \in \Omega, \\
u=0, \hspace{2 cm} x \in \partial \Omega,  
\end{array}\right.
\leqno{(P_2)}
$$
where $0<q<1<p, N<\frac{N+2}{N-2}$ if $N>2$, $u_+=\max\{u,0\}$ and $u_-=\max\{-u,0\}$.

\begin{theorem} \label{mainthm} Assume that $(H),(a_1),(a_2),(f_1)-(f_3)$ and $(g)$ hold. Then, problem $(P)$ has a nontrivial solution.
	
\end{theorem}

The Theorem \ref{mainthm} completes the study made in the papers above mentioned involving the nonlocal problem, because we are considering a new class of function $a$ and the method used in the proof, that involves dynamical methods, is new for this class of problem. Moreover it is important point out that our main result includes the local case, that is, the case where $a(x,t)=1$ for all $(x,t) \in \overline{\Omega} \times \mathbb{R}$.  We would like point out the in the proof of Theorem \ref{mainthm} it was used some ideas found in \cite{Q}.
 
Before concluding this section, we would like to show two examples where we can apply Theorem \ref{mainthm}. In both of them, we will consider the function $a:\overline{\Omega} \times \mathbb{R} \to \mathbb{R}$ of the type
$$
a(x,t)=1+B(x)h(t), \quad \forall (x,t) \in \overline{\Omega} \times \mathbb{R},
$$
where $B$ and $h$ are $C^{1}$ bounded functions,  $\displaystyle \inf_{(x,t) \in \overline{\Omega} \times \mathbb{R}}B(x)h(t)>-1$, and there is $K>0$ such that $h(t)=0$ for $|t| \geq K$.\\

\noindent {\bf Example 1: $N=2$}
$$
\left\{\begin{array}{l}
-a\left(x,\int_{\Omega}e^{|u|^{\xi}}\,dx \right)\Delta u =|u|^{p-2}ue^{|u|^{\tau}},  \quad x \in \Omega \\
u=0, \hspace{2 cm} x \in \partial \Omega,  
\end{array}\right.
\leqno{(P)}
$$ 
where $1<2\tau< \xi<2$ and $p \in (2,+\infty)$. \\

\noindent {\bf Example 2: $N \geq 3$}
$$
\left\{\begin{array}{l}
-a\left(x,\int_{\Omega}|u|^{q}\,dx \right)\Delta u =|u|^{p-1}u+|u|^{r-1}u,  \quad x \in \Omega \\
u=0, \hspace{2 cm} x \in \partial \Omega,  
\end{array}\right.
\leqno{(P)}
$$ 
where $ q \in [1,2^*]$ and  $p,r \in (1,q/2)$.

\section{Local Existence}

In this section, we are going to apply the semigroup theory to show the local existence of the solution of problem $\eqref{1}$. In the sequel, we will denote by 
$$
X=L^{2}(\Omega),\;\; A=-\Delta, 
$$
with 
$$ 
D(A)=H^{2}(\Omega)\cap H_{0}^{1}(\Omega). 
$$
By Lions and Magenes \cite{LM}, we know that 
$$
 D(A^{\frac{1}{2}})=H_{0}^{1}(\Omega).
$$
Hereafter $\|\;\;\;\|$ denotes the usual norm in $H_{0}^{1}(\Omega)$, that is, 
$$
\|u\|=\left(\int_{\Omega}|\nabla u|^{2}\,dx \right)^{\frac{1}{2}}.
$$

Our first step is to prove that the operator $\hat{f} :H_{0}^{1}(\Omega) \rightarrow X$ given by 
$$
\hat{f}(u)(x)=f(u(x)), \quad \forall x \in \Omega,
$$
is a locally Lipschitz, i.e, 
	\begin{equation}\label{L12}
	||\hat{f}(u)-\hat{f}(v)||_{X} \leq L_0||u-v||,\;\;\forall u, v\in V, 
	\end{equation}
	where $V$ is a bounded subset of $H_{0}^{1}(\Omega)$ defined by 
	$$
	V:=\{u \in H_{0}^{1}(\Omega) \,|\;\; ||u|| \leq \delta\}, 
	$$
for $\delta>0$. In what follows, for the reader's convenience we will show (\ref{L12}) for $N=2$, the case $N \geq 3$ can be done of a similar way.  Since $f \in C^{1}(\mathbb{R},\mathbb{R})$,
	$$
	 f(u)-f(v)=\int_{0}^{1}\frac{d}{ds}f(v+s(u-v))\,ds=(u-v)\int_{0}^{1}f'(v+s(u-v))\,ds,
	$$
then by $(f_1)$,
	$$
	 |f(u)-f(v)|\leq |u-v|\int_{0}^{1}|f'(v+s(u-v))|\,ds\leq |u-v|c_{\beta}e^{\beta (|v|+|u|)^2}.
	 $$
	By integration with respect to $x$ and using H\"older and Sobolev embedding, we get 
	$$
	||\hat{f}(u)-\hat{f}(v)||_{X}\leq c\|u-v\|_{L^{4}(\Omega)}\left(\int_{\Omega}c_{\beta}e^{4\beta (|v|+|u|)^{2}}\,dx\right)^{1/4} \leq c_1\|u-v\|\left(\int_{\Omega}c_{\beta}e^{4\beta (|v|+|u|)^{2}}\,dx\right)^{1/4}.
	$$
For $\|u\|,\|v\| \leq \delta$, we have that
$$
\||u|+|v|\| \leq 2\delta,
$$
and so, choosing  $\beta \in (0, \pi/4\delta^{2})$ and applying $\eqref{T}$ we deduce that $\eqref{L12}$ holds. By condition $(g)$, we also have that the operator $\tilde{\Psi}: H_0^{1}(\Omega) \to X$ given by 
$$
\tilde{\Psi}(u)(x)=\Psi\left(x,u,\int_{\Omega}g(u)\,dx \right), \quad \forall x \in \Omega,
$$
where $\Psi$ was given in (\ref{psi} ), is  locally Lipschitz, i.e, 
\begin{equation}\label{L1}
||\tilde{\Psi}(u)-\tilde{\Psi}(v)||_{X} \leq L_1||u-v||,\;\;\forall u,v\in V, 
\end{equation}
for some $L_1>0$. From this, the operator $\Phi: H_0^{1}(\Omega) \to X$ given by
\begin{equation} \label{tildef}
\Phi(u)=\hat{f}(u)+\tilde{\Psi}(u)
\end{equation}
is locally Lipschitz, i.e, 
\begin{equation}\label{L1*}
||\Phi(u)-\Phi(v)||_{X} \leq L||u-v||,\;\;\forall u,v \in V,
\end{equation}
for some $L>0$.

Arguing as in \cite[Theorem 3.3.3.]{D}, problem $\eqref{1}$ can be converted into the initial value problem for the first-order abstract evolution equation 
\begin{equation}\label{eq1}
\left\{
\begin{array}{l}
\frac{du}{dt}+Au=\Phi(u),\;\;0< t< T, \\
u(0)=u_{0}.
\end{array}\right.
\end{equation}
Hence, the existence of solution of $\eqref{eq1}$ is equivalent to look for fixed point of the operator 
\begin{equation}\label{in1}
G(u)(t)=e^{-At}u_{0}+\int_{0}^{t}e^{-A(t-s)}\Phi(u(s))\,ds. 
\end{equation}

For $\delta >0$, we set 

$$
S=\left\{u\in C\left([0,T], H_{0}^{1}(\Omega) \right)\big|\;\; \max_{t\in [0,T]}||u(t)-u_0|| \leq \delta \right\}.
$$
It is clearly that $Y=C\left([0,T], H_{0}^{1}(\Omega))\right)$ is a Banach space for the norm 
$$
||u||_{Y}=\max_{t\in [0,T]}||u(t)||,
$$
from where it follows that $S$ is a complete metric space with the metric $d:S \times S \to \mathbb{R}$ given by 
$$
d(u,v)=\|u-v\|_{Y}.
$$

In the sequel we are going  to show that $G$ maps $S$ into itself, and $G$ strict contraction for $T$ small enough. First of all, we would like point out that $G(u):[0,T] \to H_0^{1}(\Omega)$ is a continuous function for each $u \in S$, see proof of \cite[Lemma 3.3.2]{D}. Next, we will prove that $G(S) \subset S$. Have this in mind,   note that 
\begin{equation} \label{eqq1}
||G(u)(t)-u_0|| \leq ||e^{-tA}u_{0}-u_0||+\int_{0}^{t}||e^{-A(t-s)}\Phi(u(s))||\,ds.
\end{equation}
From \cite[Theorems 1.3.4 and 1.4.3]{D} we can estimate  $||e^{-tA}u_{0}-u_0||$ and $||e^{-A(t-s)}\Phi(u(s))||$ as follows:
\begin{equation}\label{eqq2} ||e^{-tA}u_{0}-u_0|| \leq \delta/2, \quad \forall t \in [0,T]
\end{equation}
and
\begin{equation}\label{eqq3} ||e^{-(t-s)A}\Phi(u(s))||=||A^{\frac{1}{2}}e^{-(t-s)A}\Phi(u(s))||_{X}\leq  \frac{M||\Phi(u(s))||_{X}}{(t-s)^{1/2}}.
\end{equation}
Let $B_*=\displaystyle \max_{t\in [0,T]} ||\Phi(u(t))||_{X}$ and choose $T$ small enough such that 
$$
2MB_*T^{1/2}\leq \frac{\delta}{2}.
$$
From $\eqref{eqq1}$, $\eqref{eqq2}$ and $\eqref{eqq3}$,  
$$ 
||G(u)(t)-u_0|| < \frac{\delta}{2} +MB_*\int_{0}^{T}r^{-1/2}\,dr=\frac{\delta}{2} + 2MB_*<\delta, \quad \forall t \in [0,T],
$$
showing that $G(S) \subset S$.

New, we will prove that of $T$ is small enough, then $G:S \to S$ is a contraction. In fact, for any $u,v\in S$ for $0\leq t\leq T$, we have 
$$ ||G(u)(t)-G(v)(t)||\leq \int_{0}^{t}||A^{\frac{1}{2}}e^{-A(t-s)}||_{X}||\Phi(u(s))-\Phi(v(s))||_{X}\,ds\leq 2MLT^{1/2}||u-v||_{Y}.$$
Thus, if we assume 
$$ T< \frac{1}{8M^{2}L^{2}}, $$
we get
$$
d(G(u),G(v))=||G(u)-G(v)|| \leq \frac{1}{2}||u-v||_{Y}=\frac{1}{2}d(G(u),G(v)), 
$$
showing the desired result. Consequently, by Banach Fixed Point Theorem, $G$ has a unique fixed point $u$ in $S$, which is a {\it mild solution} of $\eqref{in1}$.\\

As a consequence of the results found in \cite[Theorem 3.2.2]{D}, we have the following regularity result.

\begin{lemma}\label{Rem1}
The mild solution $u:[0,T] \rightarrow H^{1}_0(\Omega)$ of (\ref{1}) given by 
\begin{equation}\label{in1*}
u(t)=e^{-At}u_{0}+\int_{0}^{t}e^{-A(t-s)}\Phi(u(s))\,ds 
\end{equation}
is continuous from $[0,T] \rightarrow H^{1}_0(\Omega)$ and locally H\"older continuous from $(0,T) \rightarrow H^{1}_0(\Omega)$. Hence, $t\mapsto \Phi(u(t))$ is locally H\"older continuous on $(0,T)$, and so, $u$ is a strong unique solution of (\ref{1}) in $(0,T)$.
\end{lemma}

Next, we will show two important results involving the solution $u$  that are crucial in our approach. The first result shows that the continuous dependence of the solutions also hold with the problem (\ref{1}), while the second one establishes the asymptotic behavior of the solutions. From now on, $T(u_0)$ the maximal existence time of this solution in  $H^{1}_{0}(\Omega)$, $J(u_0) := [0, T(u_0))$ and $\mathring{J}(u_0) := (0, T(u_0))$.

\begin{lemma} \label{CD} Let $u_0 \in H_0^{1}(\Omega)$ and assume that there is a sequence $(u_n) \subset H_0^{1}(\Omega)$ such that $T(u_n)=+\infty$ and $u_n \to  u_0$ in $H_0^{1}(\Omega)$. Then, for each $t \in {J}(u_0)$ there holds
	$$
	u(t,u_n) \to u(t,u_0) \quad \mbox{in} \quad H_0^{1}(\Omega) \quad  \mbox{as} \quad n \to +\infty.
	$$	
\end{lemma}
\begin{proof}
	The proof follows as in \cite[Theorem 3.4.1]{D}
\end{proof}

\begin{lemma} \label{ESTINF}
	If $u_{0}\in H^{1}_0(\Omega)$, then there are $\gamma \in (0,1/2)$ and $C>0$ such that 	
	$$
	||u_{t}(.,u_0)||_{L^{2}(\Omega)} \leq C\left[\frac{1}{t^{\frac{1}{2}-\gamma}}+\frac{1}{t^{\gamma}}\right], \quad \forall t \in \mathring{J}(u_0). 
	$$  
\end{lemma}
\begin{proof} By using \cite[Theorem 3.5.2]{D}, there are $C>0$ and $\gamma \in (0,1/2)$ such that $t \mapsto \frac{du}{dt}(t) \in X^{\gamma}$ is locally H\"older continuous on $\mathring{J}(u_0)$ and 
	$$
	||u_{t}(.,u_0)||_{X^{\gamma}} \leq C\left[\frac{1}{t^{\frac{1}{2}-\gamma}}+\frac{1}{t^{\gamma}}\right], \quad \forall t \in \mathring{J}(u_0). 
	$$  
	Since the embedding $X^{\gamma} \hookrightarrow L^{2}(\Omega)$ is continuous, we get the desired result.   
\end{proof}

As an immediate consequence of the last lemma we have the corollary below.

\begin{corollary} \label{gozero} Let $u_0 \in H_0^{1}(\Omega)$ such that $T(u_0)=+\infty$. Then, 
	$$
	\|u_t(.,u_0)\|_{L^{2}(\Omega)} \to 0 \quad \mbox{as} \quad t \to +\infty. 
	$$	
	
\end{corollary}

Before concluding this section we would like point out that the  conditions $(H),(a_0),(a_1),(f_1)-(f_2)$ and $(g)$ ensures that $\Phi'(0)=\Phi(0)=0$, then by  \cite[Theorem 5.1.1]{D},   $u=0$ is an {\it asymptotically stable equilibrium} for (\ref{eq1}), and so, for (\ref{1}).

\section{Some properties of the trajectory }

 In this section we will show some important properties of the solution $u:[0,T] \rightarrow H^{1}_0(\Omega)$ obtained in the last section. Have this in mind, we must fix some notations that will be used later on. 
 
 If $u_0 \in   H^{1}_{0}(\Omega)$ we denote by $u(t) = u(t, u_0)$ the solution (\ref{in1*}) at time $t$, $O(u_0)=\{u(t,u_0)\,:\,t \in T(u_0)\}$ and
$$
\omega(u_0)=\{u \in H^{1}_{0}(\Omega)\,:\, \exists \, t_n \to +\infty \quad \mbox{with} \quad u(t_n,u_0) \to u \quad \mbox{in} \quad H^{1}_{0}(\Omega)\}.
$$
Moreover, we denote by $D_A$ the attraction of $u=0$, i.e., 
$$
D_A = \{u_0 \in  H^{1}_{0}(\Omega); u(t, u_0) \to 0 \quad \mbox{as} \quad  t \to +\infty \}.
$$

Hereafter, we denote by $E:H_0^{1}(\Omega) \to \mathbb{R}$ the functional given by 
$$
E(u)=\int_{\Omega}\frac{1}{2}|\nabla u|^{2}\,dx-\int_{\Omega}F(u)\,dx, 
$$
where $F(s):=\int_{0}^{s}f( r)\,dr$. 

\begin{lemma} \label{der} 
	Let $u_0 \in H^{1}_0(\Omega)$ and $u$ be the solution given in (\ref{in1*}). Then, for each $T \in \mathring{J}(u_0)$ we have that  $E(u(.))\in C([0, T])\cap C^{1}((0, T))$ and 
	\begin{equation}\label{Lyp}
	\frac{d}{dt}E(u(t))=-\int_{\Omega}|u^{2}_{t}(t)|\,dx+\int_{\Omega}\Psi\left(x,u(t),\int_{\Omega}g(u(t))\,dx \right)u_t(t)\,dx, \quad \forall t \in (0,T).
	\end{equation}

\end{lemma}
\begin{proof}
The proof follows the same ideas found in Quittener and Souplet \cite[Lemma 17.5]{QS}, however for the reader's convenience we will write its proof.  Denote $E_{1}(t)=\int_{\Omega} |\nabla u(t)|^{2}\,dx$ and $G(t)=\int_{\Omega} F( u(t))\,dx$. We are going to show only the continuity of $G(t)$. Indeed, For $t,s\in (0,T)$, $s\neq t$, by using $(f_1)-(f_2)$, we obtain 
	\begin{equation}\label{Tr}
	|G(t)-G(s)|\leq C||u(t)-u(s)||_{H_{0}^{1}(\Omega)}
	\end{equation}
	where $C$ is a positive constant. 
	Therefore $E(u(.))\in C([0, T])$. From Lemma \ref{Rem1}, we have 
	\begin{equation}\label{eqy1}
	u\in C((0, T), H^{2}(\Omega) \cap H_0^{1}(\Omega)) \cap C^{1}((0,T), L^{2}(\Omega)).
	\end{equation}
	For $t\neq s$, we have 
	$$ 
	\frac{G(t)-G(s)}{t-s}=\frac{1}{t-s}\int_{\Omega} F(u(t))-F( u(s))\,dx=\int_{\Omega}\int_{0}^{1}\left(\frac{u(t)-u(s)}{t-s} \right) f( u(s)+\lambda (u(t)-u(s))\,d\lambda dx.
	$$
	Since $u \in C^{1}((0,T),L^{2}(\Omega))$, the last inequality together with the conditions $(f_1)-(f_2)$ leads to 
	$$ 
	\lim_{s \to t}\frac{G(t)-G(s)}{t-s}=\int_{\Omega} u_{t}(t)f(u(t))\,dx.
	$$
	On  the other hand,  by using integration by parts we obtain 
	\begin{eqnarray*}
		\frac{E_1(t)-E_1(s)}{t-s}&=&\frac{1}{t-s}\int_{\Omega} \nabla (u(t)-u(s)).\nabla (u(t)+u(s))\,dx \\
		&=& -\int_{\Omega} \left(\frac{u(t)-u(s)}{t-s}\right)\Delta (u(t)+u(s))\,dx\rightarrow -2\int_{\Omega} u_{t}(t)\Delta u(t)\,dx, \;\text{as}\,\,s\rightarrow t.
	\end{eqnarray*}
	Consequently $E(u(.))\in C^{1}((0, T))$ and 
	$$ 
	\frac{d}{dt} E(u(t,u_0))=\int_{\Omega} (-\Delta u-f( u))u_{t}\,dx=-\int_{\Omega}|u_{t}(t)|^{2}\,dx+\int_{\Omega}\Psi\left(x,u(t),\int_{\Omega}g(u(t))\,dx \right)u_t(t)\,dx.
	$$
\end{proof}

As a consequence of the last result, we have the following corollary.

\begin{corollary} \label{1M}
	Let $u_0 \in H^{1}_0(\Omega)$.  Then 
	$$ 
	\int_{\Omega}g(u(t))\,dx \geq K \Rightarrow V'_{u_{0}}(t)\leq 0,
	$$
	where $K$ was given in $(a_2)$ ( see also (\ref{eq1o})), $u(t)=u(t,u_0)$ and $V_{u_0}(t)=E(u(t))$.	
	
\end{corollary}
\begin{proof} From Lemma \ref{der},  
	$$
	V'_{u_{0}}(t)=-\int_{\Omega}|u^{2}_{t}(t)|\,dx+\int_{\Omega}\Psi\left(x,u(t),\int_{\Omega}g(u(t))\,dx \right)u_t(t)\,dt, \quad \forall t \in \mathring{J}(u_0).
	$$	 
	Hence, by (\ref{eq1o}),
	$$
	\int_{\Omega}g(u(t))\,dx \geq K \; \Rightarrow \; V'_{u_{0}}(t)=-\int_{\Omega}|u^{2}_{t}(t)|\,dx\leq 0,
	$$ 	
	proving the lemma.
\end{proof}

\begin{lemma} \label{PS}
	The functional $E$ satisfies the $(PS)$ condition. 
\end{lemma}
\begin{proof} Let $\{u_n\} \subset H_0^{1}(\Omega)$ be a $(PS)_d$ sequence for $E$, that is, 
	$$
	E(u_n) \to d \quad \mbox{and} \quad E'(u_n) \to 0.
	$$	
	By $(f_3)$, there is $C>0$ such that
	\begin{equation} \label{est1}
	E(u_n)-\frac{1}{2+\gamma}E'(u_n)u_n \geq \left(\frac{1}{2}-\frac{1}{2+\gamma}\right)\|u_n\|^{2}-C, \quad \forall n \in \mathbb{N}.
	\end{equation}
	On the other hand, as $\{u_n\}$ is a $(PS)_d$ sequence, there is $n_0 \in \mathbb{N}$ such that
	\begin{equation} \label{est2}
	E(u_n)-\frac{1}{2+\gamma}E'(u_n)u_n \leq d +1 +\|u_n\|, \quad \forall n \geq n_0.
	\end{equation}
	From (\ref{est1})-(\ref{est2}), $\{u_n\}$ is bounded in $H_0^{1}(\Omega)$, and so, for some subsequence, still denoted by itself, there is $u \in H_0^{1}(\Omega)$ such that 
	$$
	u_n \rightharpoonup u \quad \mbox{in} \quad H_0^{1}(\Omega),
	$$
	$$
	u_n \to u \quad \mbox{in} \quad L^{p}(\Omega), \quad \forall p \in [1,+\infty) \quad \mbox{if} \quad N=2 \quad \mbox{and}  \quad \forall p \in [1,2^*) \quad \mbox{if} \quad N \geq 3,
	$$
	and
	$$
	u_n(x) \to u(x) \quad \mbox{a.e. in} \quad \Omega.
	$$
	The above limits combined with $(f_1)-(f_2)$ ensure that 
	$$
	\int_{\Omega}f(u_n)u_n\,dx \to \int_{\Omega}f(u)u\,dx
	$$	
	and
	$$
	\int_{\Omega}f(u_n)u\,dx \to \int_{\Omega}f(u)u\,dx.
	$$	
	Recalling that $E'(u_n)u_n=E'(u_n)u=o_n(1)$, we derive that 
	$$
	\|u_n-u\|^{2}=\int_{\Omega}f(u_n)u_n\,dx-\int_{\Omega}f(u)u\,dx+o_n(1),
	$$	
then
	$$
	\|u_n-u\|^{2}=o_n(1),
	$$	
	implying that $u_n \to u$ in $H_0^{1}(\Omega)$, finishing the proof. 
\end{proof}

\begin{lemma}\label{lm1}
	Let $ u_{0}	\in \partial D_{A}$. Then the function $V_{u_0} : J(u_{0})\rightarrow \mathbb{R}$ given in Corollary \ref{1M} is bounded.
\end{lemma}

\begin{proof}
First of all, recall that 
\begin{equation} \label{WR1}
\|u(t)\|^{2}=\int_{\Omega}f(u(t))u(t)\,dx+\int_{\Omega}\Psi\left(x,u(t),\int_{\Omega}g(u(t))\,dx \right)u(t)\,dx-\int_{\Omega}u(t)u_t(t)\,dx.
\end{equation}
If $t \geq \delta$, by Lemma \ref{ESTINF}, we have that 
\begin{equation} \label{WR2}
\|u_t(t)\|_{L^{2}(\Omega)} \leq C,
\end{equation}
for some $C>0$. Moreover, if 
$$
\int_{\Omega}g(u(t))\,dx \leq K
$$
by $(H)$
\begin{equation} \label{WR2}
\int_{\Omega}f(u(t))u(t)\,dx \leq C.
\end{equation} 
Then, from (\ref{NEWEQ1})
$$
\int_{\Omega}|u(t)|^{2+\gamma}\,dx \leq C.
$$
Since $\Omega$ is a bounded domain and $\gamma>0$, the last inequality implies that there is $C>0$ such that
\begin{equation} \label{WR3}
\int_{\Omega}|u(t)|\,dx \leq C \quad \mbox{and} \quad \int_{\Omega}|u(t)|^{2}\,dx \leq C. 
\end{equation}
Thus, 
\begin{equation} \label{IMPEST}
t \geq \delta \quad \mbox{and} \quad \int_{\Omega}g(u(t))\,dx \leq K \Rightarrow  \|u(t)\| \leq C_*,
\end{equation}
for some constant $C_*>0$. On the other hand, since $u \in C([0,\delta],H^{1}_0(\Omega))$, increasing $C_*$ if necessary, we derive that
\begin{equation} \label{ZRR1}
\|u(t)\| \leq C_*/2, \quad \forall t \in [0, \delta].
\end{equation}

Let us also assume that $ \|u_{0}\|\leq C_*/2$. Thanks to condition $(g)$, we have that  $ E : H_{0}^{1}(\Omega)\rightarrow \mathbb{R}$ is bounded on bounded sets, then  there is $M>0$ such that $|E(u)| < M$ for $\|u\|\leq C_*$. Since $u_{0}\in \partial D_{A}$, we may choose $u_{n} \in D_{A}$, $u_{n}\rightarrow u_{0}$ in $H_0^{1}(\Omega)$. Thus, there is $n_0 \in \mathbb{N}$ such that  $u_{n}\in B_{C_*}(0)$ and $u(t, u_{n})\rightarrow 0$  as $t\rightarrow +\infty$ for all $n \geq n_0$. 
\begin{claim}
	\begin{equation}\label{lm}
	|V_{u_{n}}(t)| \leq M, \quad \forall t \in \mathbb{R}^+ \quad \mbox{and} \quad \forall n \geq n_0.
	\end{equation}
\end{claim}
Indeed, if the above claim is not true, there are some $n \geq n_0$ and $t_n >0$ such that $|V_{u_{n}}(t_{n})|> M$.  Note we can assume that $t_n >\delta$, otherwise we must have
$t_n \in [0, \delta]$ for some subsequence, and so,  
$$
u(t_n,u_n) \to u(t_0,u_0) \quad \mbox{in} \quad H^{1}_0(\Omega)
$$
for some $t_0 \in [0,\delta]$. Hence, by (\ref{ZRR1}), 
$$
\|u(t_n,u_n)\| \leq C_*,
$$
for $n$ large,  then $|V_{u_{n}}(t_{n})|= |E(u(t_n,u_n))| < M$, which is absurd.

If $V_{u_{n}}(t_{n}) >  M$, we set
$$
s_n=\min\{t>0\,:\,V_{u_{n}}(s) >  M, \quad \forall s \in (t,t_n] \}.
$$
Since $u(t, u_{n})\rightarrow u_0$  as $t\rightarrow 0$, $s_n$ is well defined and by continuity of $V_{u_n}$, we deduce that
$$
V_{u_{n}}(s_{n})=M. 
$$ 
Arguing as above, we can also assume that $s_n >\delta$, then 
$$
s_n >\delta \quad \mbox{and} \quad 	\int_{\Omega}g(u(t))\,dx > K, \quad \forall t \in [s_n,t_n],
$$
and so, $V_{u_{n}}'(t)\leq 0$ for all $t\in [s_n, t_{n}]$. Therefore,   
$$
V_{u_{n}}(t_n) \leq  V_{u_{n}}(s_{n})=M,
$$ 
which is a contradiction. If $V_{u_{n}}(t_{n})<- M$, we set  
$$
d_{n}:= \max\{t> 0;\, V_{u_{n}}(s)< -M, \; \forall s\in [t_n, t) \}, 
$$
which is well defined because $u(t, u_{n})\rightarrow 0$  as $t\rightarrow +\infty$. By continuity of $V_{u_n}$ we have that $V(d_{n})=-M$. Thereby, 
$$
t_n > \delta \quad \mbox{and} \quad	\int_{\Omega}g(u(t))\,dx > K, \quad \forall t \in [t_n,d_{n}],
$$
and so, 
$$
V_{u_{n}}'(t)\leq 0 \quad \forall t\in [t_n,d_{n}].
$$
From this, 
$$
-M=V_{u_{n}}(d_{n}) \leq V_{u_{n}}(t_{n}),
$$ 
obtain again a new contradiction. This proves the Claim \ref{lm}. Thanks to Lemma \ref{CD}, we easily  get $|V_{u_{0}}(t)|\leq M$ for any $t\in J(u_{0}),$ finishing the proof of the lemma.  

\end{proof}

Now  we are going to show a strong version of Lemma \ref{1M}.
\begin{lemma}\label{stl}
Let $ u_{0}	\in \partial D_{A}$. Then there are $\delta>0$ and $K_* >K$ such that 
$$
\;||u(t)|| \geq K_* \; \Rightarrow \; V'_{u_{0}}(t)< -\delta, 
$$
where $K$ was given in $(a_2)$.

\end{lemma}
\begin{proof} Assume by contradiction that the lemma does not hold. Then, for each $n \geq K$ must exist $t_n>0$ such that
$$
\|u(t_n,u_0)\|\geq n \quad \mbox{and} \quad V'_{u_0}(t_n)>-1/n, \quad \forall n \in \mathbb{N}.
$$
The first inequality implies that $t_n \to +\infty$, then by (\ref{IMPEST}),
$$
t_n \geq \delta \quad \mbox{and} \quad \int_{\Omega}g(u(t_n,u_0))\,dx > K, 
$$
for $n$ large enough. Since 
$$
E'(u)v=\int_{\Omega}\nabla u\nabla v\,dx-\int_{\Omega}f(u)v\,dx=-\int_{\Omega}\left(u_{t}+\Psi\left(x,u(t),\int_{\Omega}g(u(t))\,dx \right)\right)v\,dx,
$$
by (\ref{eq1o}), 
$$
\|E'(u_n)\|^{2} \leq \int_{\Omega}|u_{n,t}|^{2}\,dx,
$$
for $n$ large enough, where $u_n=u(t_n,u_0)$. On the other hand, we also have
$$
-\frac{1}{n}<V_{v_n}'(t_n)=-\int_{\Omega}|u_{n,t}|^{2}\,dx\leq -\int_{\Omega}|u_{n,t}|^{2}\,dx+o_n(1)\leq -\|E'(u_n)\|^{2} ,
$$
leading to
$$
\|E'(u_n)\|^{2}\leq \frac{1}{n}.
$$
By Lemma \ref{lm1} the sequence  $\{E(u_n)\}$ is bounded, and so,  $\{u_n\}$ is an unbounded $(PS)$ sequence for $E$, which contradicts Lemma \ref{PS}.  
\end{proof}	

\begin{lemma} \label{TINF} Let $u_0 \in \partial D_A $. Then,  $T(u_0)=+\infty$.
	
\end{lemma}
\begin{proof} To begin with, we make the following claim:
\begin{claim} \label{ESTIMATIVAPSI} There is $C>0$ such that
$$
\int_{\Omega}\left|\Psi\left(x,u(t),\int_{\Omega}g(u(t))\,dx \right)\right|^{2}\,dx \leq C, \quad \forall t \in J(u_0).
$$	
\end{claim}	
\noindent Indeed, if $\int_{\Omega}g(u(t))\,dx >K$, by (\ref{eq1o}) we have that
$$
\Psi\left(x,u(t),\int_{\Omega}g(u(t))\,dx \right)=0.
$$ 
On the other hand, if $\int_{\Omega}g(u(t))\,dx \leq K$, $(H)$ and $(a_1)-(a_2)$ combine to give  
$$
\int_{\Omega}\left|\Psi\left(x,u(t),\int_{\Omega}g(u(t))\,dx \right)\right|^{2} \leq C,
$$ 
for some constant $C$ that does not depend on $t$. This proves the claim. 

Since the Claim \ref{ESTIMATIVAPSI} is true, we can argue as in \cite[Theorem 1 and Remark 2]{Q} to obtain  
	\begin{equation*} \label{B1}
	\int_{\Omega}|u(t)|^{2}\,dx \leq c_1e^{c_2t}, \quad \forall t \in J(u_0)
	\end{equation*}	
	and
	$$
	\int_{\Omega}(f(u(t))u(t)-2F(u(t)))\,dx \leq c_1e^{c_2t}, \quad \forall t \in J(u_0),
	$$
	for some positive constants $c_1,c_2>0$. Then, by $(f_3)$, 
	$$
	\int_{\Omega}F(u(t))\,dx \leq c_3e^{c_2t}, \quad \forall t \in J(u_0).
	$$
	This together with the boundedness of $V_{u_0}$ leads to 
	\begin{equation*} \label{B2}
	\|u(t)\|^{2}\leq c_4+c_3e^{c_2t}, \quad \forall t \in J(u_0),
	\end{equation*} 
	which completes the proof.
\end{proof}

\begin{proposition}\label{tth1}
Let $u_{0}\in \partial D_{A}$. Then the orbit $O (u_{0})$  is bounded in $H_{0}^{1}(\Omega)$.
\end{proposition}
\begin{proof} The proof follows the same ideas found in \cite[Theorem 2]{Q}, however as in some points the estimate is different we will write the proof.
Assume the contrary, i.e,  $\displaystyle\limsup_{t\to \infty}||u(t)||=\infty$ and put $M'=\displaystyle\liminf_{t\to \infty}||u(t)||$. If $M'=\infty$, then $||u(t)||\geq K_*$ for $t\geq t_{0}$. Hence $V'(t)< -\delta$ for $t\geq t_{0}$, which contradicts Lemma \ref{lm1}. If $M'< \infty$ we shall derive a contradiction by using the idea found in \cite{Q}. Choose $R> \max\{M', K_*\}+1$. Then there exist sequences $\{t_{n}\}$ and $\{T_{n}\}$, such that $t_{n}< T_{n}<t_{n+1}$, $||u(t_{n})||=R$, $||u(T_{n})||=n $ and $R< ||u(t)||< n$ for $t\in (t_{n}, T_{n})$. Since $V'(t)< -\delta $ for $t\in [t_{n}, T_{n}]$ and the function $V$ is bounded by Lemma \ref{lm1}, then by integrating over $(t_{n}, T_{n}) $ we get $V(T_{n})-V(t_{n})< -\delta (T_{n}-t_{n})$, which is equivalently to $ V(t_{n})-V(T_{n})> \delta (T_{n}-t_{n})> 0$. Thus, $C_*=\frac{2M_*}{\delta}> T_{n}-t_{n}> 0$, where $M_*=\displaystyle \sup_{t \in J(u_0)}V(t)$. \\
	Now using the variation of constants formula as in above we now show $T_{n}-t_{n}> c$ for some $c> 0$. For $t> s\geq 0$ we have 
	\begin{equation}
	u(t)=e^{-A(t-s)}u(s)+\int_{s}^{t}e^{-A(t-\tau)}\Phi(u(\tau))\,d\tau 
	\end{equation}
where $\Phi$ was given in (\ref{tildef}). From \cite[Theorem 1.3.4 and Theorem 1.4.3]{D} we obtain the following estimates 
	\begin{equation}\label{er}
	||u(t)||\leq M||u(s)||+M\int_{s}^{t}\frac{||\Phi(u(\tau))||_{L^{2}(\Omega)}}{(t-\tau)^{1/2}}\,d\tau,
	\end{equation}
	\begin{equation}\label{es2}
	||u(t)||_{\beta}\leq \frac{M}{(t-s)^{\beta-\frac{1}{2}}}||u(s)||+M\int_{s}^{t}\frac{||\Phi(u(\tau))||_{L^{2}(\Omega)}}{(t-\tau)^{\beta}}\,d\tau,
	\end{equation}
	where $M$ is a positive constant and $||.||_{\beta}$ denotes the norm in $D(A^{\beta})$ for $\beta \in (\frac{1}{2},1)$. \par 
	If we take $t=t_{n}+c$ in $\eqref{er}$ for some $c>0$ then by the above assumptions we can get the existence of a sequence $s_{n}\in (t_{n}, t_{n}+c)$ such that $||u(s_{n})||\leq M'+1$. Thus 
	\begin{eqnarray*}
		||u(t_{n}+c)|| & \leq& M||u(s_{n})||+M\int_{s_{n}}^{t_{n}+c}\frac{||\Phi(u(\tau))||_{L^{2}(\Omega)}}{(t_{n}+c-\tau)^{1/2}}\,d\tau\\
		& \leq &
	M(1+M')	+M \int_{t_{n}}^{t_{n}+c}\frac{||\Phi(u(\tau))||_{L^{2}(\Omega)}}{(t_{n}+c-\tau)^{1/2}}\,d\tau
	\end{eqnarray*}
	by the continuity of $\tau \rightarrow \frac{||\Phi(u(\tau))||_{L^{2}(\Omega)}}{(t_{n}+c-\tau)^{1/2}}$ we get  
	$$ \int_{t_{n}}^{t_{n}+c}\frac{||\Phi(u(\tau))||_{L^{2}(\Omega)}}{(t_{n}+c-\tau)^{1/2}}\,d\tau < c^{1/2}||\Phi(u(t_{n}))||_{L^{2}(\Omega)}+\varepsilon,\; \forall\varepsilon >0,  $$
	Using assumption  $||u(t_{n})||=R$ with Trundiger-Moser inequality we obtain 
	$$
	||\Phi(u(t_{n}))||_{L^{2}(\Omega)}\leq L||u(t_{n})||=LR,
	$$
leading to
	$$
	||u(t_{n}+c)||\leq C, \quad \forall n \in \mathbb{N}.
	$$
	Therefore 
	$ ||u(t)||\leq C$ if $t\in [t_{n}, t_{n}+c]$. If we assume $T_{n}  \leq t_{n}+c$ then 
	$n=||u(T_{n})|| \leq C$ but this a contradiction as $n\rightarrow \infty$. Hence we get $T_{n} > t_{n}+c$. Let $\theta \in (0, c)$ by using $\eqref{es2}$, we obtain
	
	$$ ||u(t_{n}+\theta)||_{\beta}\leq M
	\theta^{-(\beta-\frac{1}{2})}||u(t_{n})||+M\int_{t_{n}}^{t_{n}+\theta}\frac{||\Phi(u)||_{L^{2}}}{(t_{n}+\theta-\tau)^{\beta}}\,d\tau  $$
	Since the right hand side is bounded and $D(A^{\beta})$ is compactly embedded into $D(A^{\frac{1}{2}})=H^{1}_0(\Omega)$ \cite[Theorem 3.3.6]{D} . Then we can extract from $\{u(t_{n}+\theta)\}$ a convergent subsequence in $H_{0}^{1}(\Omega)$. Thus, $u(t_{n}+\theta)\rightarrow u_{1}\in \omega(u_{0})$. Since $u(., u_{1}) :[0,\infty) \rightarrow H^{1}_0(\Omega)$ is bounded on $[0, C_*]$ and by using continuous depends we have $u(., u(t_{n}+\theta, u_{0}))$ converge to $u(., u_{1})$ uniformly on $[0, C_*] $, We arrive to a contradiction with 
	$$
	||u(T_{n}-t_{n}-\theta, u(t_{n}+\theta, u_{0}))||=||u(T_{n})||=n\rightarrow \infty.
	$$ 
\end{proof}

 \section{Some remarks about the global existence and blow-up of the solution }
 
The main goal this section is showing some conditions that guarantee the global existence of the solution, and also when the blow-up phenomena holds.

\begin{proposition} \label{local} There is $r>0$ such that if $\|u_0\| < r$, then $T(u_0)=+\infty$. Moreover, $0 \in \mathring{D}_A=int \, D_A$.
\end{proposition}	 
\begin{proof} This results is an immediate consequence of the fact that $u=0$ is an asymptotically stable equilibrium for (\ref{1}).
\end{proof}

\begin{corollary} \label{wuo} For all $u_0 \in H^{1}_0(\Omega)$, the set $\omega(u_0)$ is formed by stationary points. Moreover, if $u_0 \in \partial D_A$ we also have that $0 \notin \omega(u_0)$.  

\end{corollary}	
\begin{proof} Let $u \in \omega(u_0)$. Then, there is $t_n \to +\infty$ such that $u_n=u(t_n,u_0) \to u$ in $H^{1}_0(\Omega)$. Since
$$
\int_{0}(u_n)_tv\,dx+\int_{\Omega}\nabla u_n \nabla v\,dx=\int_{\Omega}f(u_n)v\,dx+\int_{\Omega}\Psi\left(x,u_n,\int_{\Omega}g(u_n)\,dx \right)v\,dx, \quad \forall v \in H^{1}_0(\Omega),
$$
taking the limit $n \to +\infty$ and using $(H),(g),(a_1)-(a_2)$ and $(f_1)-(f_2)$, together with the limits $u_n \to u$ in $H^{1}_0(\Omega)$ and $(u_n)_t=u_t(t_n) \to 0$ in $L^{2}(\Omega)$, we get
$$
\int_{\Omega}\nabla u \nabla v\,dx=\int_{\Omega}f(u)v\,dx+\int_{\Omega}\Psi\left(x,u,\int_{\Omega}g(u)\,dx \right)v\,dx, \quad \forall v \in H^{1}_0(\Omega),
$$
showing $u$ is a stationary point. Now we are going to show that $0 \notin \omega(u_0)$. If we assume by contradiction that $0 \in \omega(u_0)$, since $u=0$ is an asymptotically stable equilibrium for (\ref{1}), the continuous dependence of the solutions will imply that $u_0 \in \mathring{D}_A$, which is a contradiction.
\end{proof}

 \begin{proposition}\label{Prp1}
 Let $u_0 \in H^{1}_0(\Omega)$ with 
 $$
\|u_0\| > K_* \quad \mbox{and} \quad E(u_0)<\min\left\{-\hat{M}=\displaystyle \inf_{\|u\|=K_*}E(u),-\frac{K}{2+\gamma}\right\},
 $$
where $K_*$ was given in Lemma \ref{stl}.  Then $T(u_0)<+\infty$.

 \end{proposition}
 \begin{proof} To begin with, we claim that 
\begin{equation} \label{NESEST}
 \|u(t)\| > K_*, \quad \forall t \in J(u_0).
\end{equation}
Indeed, otherwise there is $T_1 \in J(u_0)$ such that 
 $$
 \|u(T_1)\|=K_* \quad \mbox{and} \quad \|u(t)\| \geq K_* \quad \forall t \in [0,T_1]. 
 $$	
 By Lemma \ref{stl}, we derive that
 $$
 V_{u_0}(t) \leq V_{u_0}(0), \quad \forall t \in [0,T_1],
 $$	
 implying that 
 $$
 E(u(T_1)) \leq E(u_0).
 $$
 Thereby, 
 $$
 E(u(T_1)) < -\hat{M},
 $$
 which is absurd, because $\|u(T_1)\|=K_*$. The inequality (\ref{NESEST}) gives that 
 \begin{equation} \label{NESESTo}
\int_{\Omega}g(u(t))> K, \quad \forall t \in J(u_0).
 \end{equation}
 Indeed, if $\int_{\Omega}g(u(t)) \leq K$ for some $t \in J(u_0)$, the equality below 
 $$
 \|u(t)\|^{2}=2\left(E(u(t))+\int_{\Omega}F(u(t))\,dx\right)
 $$
together with $(H)$ and $(f_3)$ yields  
 $$
\|u(t)\|^{2} \leq 2\left(E(u(t))+\frac{K}{2+\gamma}\right).
$$
Recalling that (\ref{NESEST}) together with Lemma \ref{stl} yields $E(u(t)) \leq E(u_0)$ for all $t \in J(u_0)$, we get 
 $$
 \|u(t)\|^{2}\leq 2\left(E(u_0)+\frac{K}{2+\gamma}\right) <0, 
 $$
 which is absurd, proving (\ref{NESESTo}).
 
 Now we are ready to prove that $T(u_0)<+\infty$. Have this in mind, we will apply the so called concavity method  as in \cite{T.C}. For every $t> 0$, let 
 \begin{equation}\label{re}
 H(t)=\frac{1}{2} \int_{0}^{t}||u(s)||^{2}_{L^{2}(\Omega)}\,ds
 \end{equation}
 then by differentiating $\eqref{re}$ we find  $H'(t)=\frac{1}{2}||u(t)||^{2}_{L^{2}(\Omega)}$, differentiating $H'$ and using (\ref{NESESTo}), we obtain 
 $$
 H''(t)=\int_{\Omega} u_{t}(t)u(t)\,dx=-\int_{\Omega} |\nabla u(t)|^{2}\,dx+\int_{\Omega} f( u(t))u(t)\,dx.
 $$
 From $(f_3)$,  
 \begin{equation}\label {eq0}
 H''(t) \geq -\int_{\Omega} |\nabla u(t)|^{2}\,dx+(2+\gamma)\int_{\Omega} F( u(t))\,dx, 
 \end{equation}
 
 Multiplying the equation $\eqref{1}$ by $u_{t}$ and by integrating over $\Omega $ and using again (\ref{NESEST}),  we find
 $$ ||u_{t}(t)||_{L^{2}(\Omega)}^{2}=\partial_{t}\left(-\frac{1}{2}\int_{\Omega} |\nabla u(t)|^{2}\,dx+\int_{\Omega} F(u(t))\,dx\right).$$
 Now integrating with respect to $t$, we obtain 
 \begin{equation}\label{eq00}
 \int_{0}^{t}||u_{t}(s)||_{L^{2}(\Omega)}^{2}\,ds=-\frac{1}{2}\int_{\Omega} |\nabla u(t)|^{2}\,dx+\int_{\Omega} F(u(t))\,dx+E(u_0).
 \end{equation}
 By substituting $\eqref{eq00}$ in $\eqref{eq0}$ we derive that 
 \begin{equation}\label{eq01}
 H''(t)\geq (2+\gamma) \int_{0}^{t}||u_{t}(s)||_{L^{2}}^{2}\,ds+\frac{\gamma}{2}\int_{\Omega} |\nabla u(t)|^{2}\,dx- (\gamma+2)E(u_0).
 \end{equation}
 Since $-(\gamma+2)E(u_0)> 0$ hence this gives the following inequalities  
 \begin{equation}\label{eq02*}
 \frac{d}{dt} \|u(t)\|_{L^{2}(\Omega)}^{2} \geq C\|u(t)\|_{L^{2}(\Omega)}^{2}, \quad \forall t \in J(u_0)
 \end{equation}
 and
 \begin{equation}\label{eq02}
 H''(t)\geq (2+\gamma) \int_{0}^{t}||u_{t}(s)||_{L^{2}(\Omega)}^{2}\,ds, \quad \forall t \in J(u_0).
 \end{equation}
 The inequality (\ref{eq02*}) ensures that $\|u(t)\|_{L^{2}(\Omega)}^{2}$ has an exponential growth, that is, there is $C>0$ such that 
 $$
 \|u(t)\|_{L^{2}(\Omega)}^{2} \geq c_1e^{c_2t}, \quad \forall t \in J(u_0).
 $$ 
 Thus, assuming by contradiction that $T(u_0)=+\infty$, we have that 
 \begin{equation}\label{eq02**}
 \|u(t)\|_{L^{2}(\Omega)} \to +\infty \quad \mbox{when} \quad t \to +\infty.
 \end{equation} 
 On the other hand, multiplying $\eqref{eq02}$ by $V$ we find 
 $$ 
 H(t)H''(t)\geq \frac{(2+\gamma)}{2} \left(\int_{0}^{t}||u_{t}(s)||_{L^{2}}^{2}\,ds\right)\left(\int_{0}^{t}||u(s)||_{L^{2}}^{2}\,ds\right).$$
Using the H\"older inequality, we get 
 $$ 
 H(t)H''(t)\geq \frac{(2+\gamma)}{2} (H'(t)-H'(0))^{2}.
 $$
 From (\ref{eq02**}) this implies that there is $T_1>0$ such that   
 \begin{equation}\label{eqq02}
 H(t) H''(t)\geq \frac{(2+\gamma)}{2} (H'(t))^{2}, \quad \forall t \geq T_1.
 \end{equation}
 Now we put $l(t)=H^{-\gamma/2}(t)$, it possible to prove after some calculations that $\eqref{eqq02}$ ensures that $l$ is a concave function for $t \geq T_1$, which is impossible because $l(t) \geq 0$ for all $t \geq 0$, and by (\ref{eq02*}), $l(t) \to 0$ when $t \to +\infty$.  Thus, $T(u_0)< \infty$, finishing the proof.
\end{proof}

As a byproduct of the study above we have the following result.

\begin{lemma} $\partial D_A \not= \emptyset$.
	
\end{lemma}
\begin{proof} Fix $v \in H^{1}_0(\Omega) \setminus \{0\}$. Then $sv \in int\, D_A$ for $s$ small enough and $E(tv)<-\hat{M}$ for $t$ large enough, then by Proposition \ref{Prp1} that $tv \not\in D_A$.  Therefore, as $[sv,tv]$ is connected with $[sv,tv] \cap D_A \not= \emptyset$ and $[sv,tv] \cap (D_A)^c \not= \emptyset$, there is $s_0 \in [s,t]$ such that 
$$
s_0v \in \partial D_A,
$$	
showing the result.	
\end{proof}

\section{Existence of a nontrivial stationary solution}

\begin{theorem} If $u_0 \in \partial D_A$, we have that $\omega(u_0) \not= \emptyset $.  Hence, by Corollary \ref{wuo}, the elliptic problem
$$
\left\{\begin{array}{l}
-\Delta u =f(u)+\Psi\left(x,u,\int_{\Omega}g(u)\,dx \right),  \quad x \in \Omega \\
u=0, \hspace{2 cm} x \in \partial \Omega.  
\end{array}\right.
\leqno{(P)}
$$
has a nontrivial solution.
\end{theorem}
\begin{proof} By Lemma \ref{TINF}, we know that $T(u_0)=+\infty$. For any sequence $\{t_n\} \subset  [0,+\infty)$ with $t_n \to +\infty$, we have $\{E(u_n)\}$ is bounded, because $V_{u_0}$ is a bounded function on $[0,+\infty)$, where $u_n=u(t_n,u_0)$. By Proposition \ref{tth1} the sequence $\{u_n\}$ is bounded, then for some subsequence there is $u_s \in H^{1}_0(\Omega)$ such that 
$$
u_n  \rightharpoonup u_s \quad \mbox{in} \quad H_0^{1}(\Omega) \quad \mbox{as} \quad n \to +\infty.
$$
On the other hand, by Lemma \ref{ESTINF}, $\|u_t\|_{L^{2}(\Omega)} \to 0$ when $t \to +\infty$, which leads to 
$$
\sup_{\|v\|\leq 1}\left[E'(u_n)v-\int_{\Omega}\Psi\left(x,u_n,\int_{\Omega}g(u_n)\,dx \right)v\,dx\right] \to 0 \quad \mbox{as} \quad n \to +\infty.  
$$
Therefore, 
\begin{equation} \label{Z1}
E'(u_n)u_n-\int_{\Omega}\Psi\left(x,u_n,\int_{\Omega}g(u_n)\,dx \right)u_n\,dx=o_n(1)
\end{equation}	
and
\begin{equation} \label{Z2}
E'(u_n)u_s-\int_{\Omega}\Psi\left(x,u_n,\int_{\Omega}g(u_n)\,dx \right)u_s\,dx=o_n(1).
\end{equation}	 	
Recalling that 
$$
\|u_n-u_s\|^{2}=E'(u_n)u_n+\int_{\Omega}f(u_n)u_n\,dx- E'(u_n)u_s-\int_{\Omega}f(u_n)u_s\,dx,
$$	
the limits below
$$
\int_{\Omega}f(u_n)u_n\,dx \to \int_{\Omega}f(u_s)u_s\,dx
$$
and 
$$
\int_{\Omega}f(u_n)u_s\,dx \to \int_{\Omega}f(u_s)u_s\,dx,
$$
lead to 
$$
\|u_n-u_s\|^{2}=E'(u_n)u_n- E'(u_n)u_s+o_n(1).
$$
This combined with (\ref{Z1})-(\ref{Z2}) gives 
$$
\|u_n-u_s\|^{2}=\int_{\Omega}\Psi\left(x,u_n,\int_{\Omega}g(u_n)\,dx \right)u_n\,dx-\int_{\Omega}\Psi\left(x,u_n,\int_{\Omega}g(u_n)\,dx \right)u_s\,dx+o_n(1).
$$	
On the other hand, a direct computation gives  
$$
\lim_{n \to +\infty}\int_{\Omega}g(u_n)\,dx=\int_{\Omega}g(u_s)\,dx,
$$
therefore, 
$$
\lim_{n \to +\infty}\int_{\Omega}\Psi\left(x,u_n,\int_{\Omega}g(u_n)\,dx \right)u_n\,dx=\int_{\Omega}\Psi\left(x,u_s,\int_{\Omega}g(u_s)\,dx \right)u_s\,dx
$$
and
$$
\lim_{n \to +\infty}\int_{\Omega}\Psi\left(x,u_n,\int_{\Omega}g(u_n)\,dx \right)u_s\,dx=\int_{\Omega}\Psi\left(x,u_s,\int_{\Omega}g(u_s)\,dx \right)u_s\, dx.
$$
Therefore, 
$$
\|u_n-u_s\|^{2}=o_n(1),
$$	
from where it follows that
$$
u_n \to u_s \quad \mbox{in} \quad H^{1}_0(\Omega),
$$
and so, $u_s \in \omega(u_0)$. Now the result follows from Corollary \ref{wuo}.  
\end{proof}


\begin{thebibliography}{20}

\bibitem{AlvesCorreaChipot} C.O. Alves, M. Chipot and F.J.S.A. Corr\^ea,{\it On a class of intermediate local-nonlocal elliptic problems }, Topol. Methods Nonlinear Anal. 49 (2017), 497-509.

\bibitem{AlvesCovei} C.O. Alves and Drago\c{s}-P\u{a}tru Covei,{\it Existence of solution for a class on nonlocal elliptic problem via sus-supersolution method}, Nonlinear Anal. Real World Appl. 23 (2015), 1-8. 	
	
\bibitem{ChipotLovat1} M. Chipot and B. Lovat, {\it Some remarks on nonlocal elliptic and parabolic problems,} Nonlinear Anal. 30 (1997) 4619-4627.

\bibitem{ChipotLovat2} M. Chipot and B. Lovat, {\it On the asymptotic behaviour of some nonlocal problems}, Positivity (1999) 65-81.

\bibitem{ChipotRodigues} M. Chipot and J.F. Rodrigues, {\it On a class of nonlocal nonlinear elliptic problems,} RAIRO Mod\'el. Math. Anal. Num\'er. 26 (1992) 447-467.

\bibitem{ChipotCorrea} M. Chipot and F.J.S.A. Corr\^ea, {\it Boundary layer solutions to functional elliptic equations}, Bull. Braz. Math. Soc. (N.S.) 40 (2) (2009) 1-13.

\bibitem{CorreaMenezesFerreira} F.J.S.A. Corr\^ea, S.D.B. Menezes and J. Ferreira, {\it On a class of problems involving a nonlocal operator,} Appl. Math. Comput. 147 (2004) 475-489.	
	
\bibitem {T.C} T. Cazenave and P. L. Lions, {\it Solutions globales d'équations de la chaleur semi linéaires}, Comm. Partial. Differ. Equations, 9 (1984), 955-978.	


\bibitem{joao} L. Gasi\'nki and Jo\~ao R. Santos J\'unior, {\it Multiplicity of positive solutions for an equation with degenerate nonlocal diffusion}, arXiv:1807.01900v1[math.AP]


\bibitem{D} D. Henry, {\it Geometric Theory of Semilinear Parabolic Equations}, Lecture Notes in Math. 840, Springer-Verlag, New York, 1981.



\bibitem{LM} J. L. Lions and E. Magenes, {\it Problemes aux limites non homogenes et applications}, Vol. I, 	Dunod, Paris, 1968.





\bibitem{M} {J. Moser}, {\it A sharp form of an inequality by N. Trudinger}, {Ind. Univ. Math. J.} (20) (1971), {1077--1092}.


\bibitem{Q} P. Quittner, {\it Boundedness of trajectories of parabolic equations and stationary solutions via dynamical methods}. Differential and Integral Equations,  7 (1994), 1547-1556.

\bibitem{Q1} P. Quittner, {\it Signed solutions for a semilinear elliptic problem}. Differential and Integral Equations,  11 (1998), 551-559.

\bibitem{QS} P. Quittner and P. Souplet,{\it  Superlinear Parabolic problems, blow-up, global existence and steady states,} Second Edition, Birkh\"auser, 2019. 



\bibitem{T} {N. S. Trudinger}, {\it On imbedding into Orlicz spaces and some application}, {J. Math Mech.} {17} (1967), {473--484}.




\bibitem {Zh} Zheng. S,{\it Nonlinear evolution equations}, Chapman \& Hall/CRC Monographs and surveys in Pure and Applied Mathematics, 133, Chapman \& Hall/CRC, Boca Raton, FL. 2004.



\end{thebibliography}
\end{document}